\documentclass{article}
\usepackage{amssymb}
\usepackage{amsmath}
\usepackage{amsthm}
\usepackage{mathabx}
\usepackage{enumerate}
\usepackage{xcolor}
\usepackage{changepage}
\usepackage{graphicx}

\usepackage{hyperref}

\theoremstyle{plain}
\everymath=\expandafter{\the\everymath\displaystyle}

\newtheorem{Lem}{Lemma}
\newtheorem{Cor}{Corollary}

\newtheorem{Thm}{Theorem}
\newcommand*{\supp}{\ensuremath{\mathrm{supp\,}}}

\newcommand*{\Ker}{\ensuremath{\mathrm{Ker\,}}}
\newcommand*{\Ann}{\ensuremath{\mathrm{Ann\,}}}
\newcommand*{\R}{\ensuremath{\mathbb{R}}}

\newcommand*{\Z}{\ensuremath{\mathbb{Z}}}

\newcommand*{\N}{\ensuremath{\mathbb{N}}}

\begin{document}

	\date{}

\author{
	L\'aszl\'o Sz\'ekelyhidi\\
	{\small\it Institute of Mathematics, University of Debrecen,}\\
	{\small\rm e-mail: \tt szekely@science.unideb.hu,}
}	
	\title{Harmonic Synthesis on Group Extensions}

	\maketitle
	
	\begin{abstract}
	Harmonic synthesis describes translation invariant linear spaces of continuous complex valued functions on locally compact abelian groups. The basic result due to L. Schwartz states that such spaces on the reals are topologically generated by the exponential monomials in the space -- in other words the locally compact abelian group of the reals is synthesizable. This result does not hold for continuous functions in several real variables as it was shown by D.I. Gurevich's counterexamples. On the other hand, if two discrete abelian groups have this synthesizability property, then so does their direct sum, as well. In this paper we show that if two locally compact abelian groups have this synthesizability property and at least one of them is discrete, then their direct sum is synthesizable. In fact, more generally, we show that any extension of a synthesizable locally compact abelian group by a synthesizable discrete abelian group is synthesizable. This is an important step toward the complete characterisation of synthesizable locally compact abelian groups.
	\end{abstract}
	
	\footnotetext[1]{The research was supported by the  the
		Hungarian National Foundation for Scientific Research (OTKA),
		Grant No.\ K-134191.}\footnotetext[3]{AMS (2000) Subject Classification: 43A45, 22D99}
	\bigskip\bigskip


	\section{Introduction}
Harmonic (or spectral) synthesis deals with uniformly closed translation invariant linear spaces of continuous complex valued functions on a locally compact Abelian group. Such a function space is called a {\it variety}. The first classical result on this field is due to Laurent Schwartz \cite{MR0023948}, which states that given any complex valued continuous function on the reals it can be uniformly approximated on compact sets by exponential polynomials, which belong to the smallest variety including the given function. This can be reformulated in the following manner: if a continuous complex valued function on the reals satisfies a system of homogeneous convolution equations corresponding to compactly supported complex Borel measures, then this function is the uniform limit on compact sets of exponential polynomials, which are also solutions of the same system of convolution equations. In particular, every such system  has exponential solutions. It is obvious, that this problem makes sense on locally compact Abelian groups using appropriate definition of exponential polynomials. One possible definition is the following: we call a continuous complex valued function on a locally compact Abelian group an exponential polynomial, if the smallest variety including the function is finite dimensional. We say that {\it spectral analysis} holds for a variety, if every nonzero subvariety contains a one dimensional subvariety. We say that a variety is {\it synthesizable}, if its finite dimensional subvarieties span a dense subspace in the variety. Finally, we say that {\it spectral synthesis} holds for a variety, if every subvariety is synthesizable. In particular, if every variety on the group is synthesizable, then we say that {\it spectral synthesis holds on the group}, or {\it the group is synthesizable}. For more about spectral analysis and synthesis on groups see \cite{MR2680008,MR3185617}.
\vskip.2cm

In \cite{MR2340978}, the authors characterized synthesizable discrete Abelian groups: spectral synthesis holds on the discrete Abelian group $G$, if and only if $G$ has finite torsion free rank. In particular, from this result it follows, that if $G$ and $H$ are synthesizable, then so is $G\times H$. Unfortunately, such a result does not hold in the non-discrete case. Namely, by the fundamental result of L.~Schwartz \cite{MR0023948}, $\R$ is synthesizable, but D.~I.~Gurevich showed in \cite{MR0390759} that spectral synthesis fails to hold on $\R\times\R$. In other words, direct product does not necessarily preserve synthesizability in non-discrete cases. However, in this paper we show, that if spectral synthesis holds on two locally compact Abelian groups, then it holds on their direct product assuming at least one of them is synthesizable. 

\section{Preliminaries}

\hspace{\parindent}  On commutative topological groups finite dimensional varieties of continuous functions are completely characterized: they consist of  the elements of the complex algebra of continuous complex valued functions generated by all continuous homomorphisms into the multiplicative group of nonzero complex numbers (\textit{exponentials}), and all continuous homomorphisms into the additive group of all complex numbers (\textit{additive functions}).  In particular, an \textit{exponential monomial} (or {\it $m$-exponential monomial}) is a function of the form
$$
x\mapsto P\big(a_1(x),a_2(x),\dots,a_n(x)\big)m(x),
$$
where $P$ is a complex polynomial in $n$ variables, the $a_i$'s are additive functions, and $m$ is an exponential. Every exponential polynomial is a linear combination of exponential monomials. The $m$-exponential monomials with $m=1$ are called {\it polynomials} (see e.g. \cite{MR2680008,MR3185617}).
\vskip.2cm

Let $G$ be a locally compact Abelian group. It is known that the dual space of $\mathcal C(G)$ can be identified with the space $\mathcal M_c(G)$ of all compactly supported complex Borel measures on $G$. This space is called the {\it measure algebra} of $G$ -- it is a topological algebra with the linear operations, with the convolution of measures and with the weak*-topology arising from $\mathcal C(G)$. On the other hand, the space $\mathcal C(G)$ is a topological vector module over the measure algebra under the action realized by the convolution of measures and functions. The annihilators of subsets in $\mathcal C(G)$ and the annihilators of subsets in $\mathcal M_c(G)$ play an important role in our investigation. For each closed ideal $I$ in $\mathcal M_c(G)$ and for every variety $V$ in $\mathcal C(G)$, $\Ann I$ is a variety in $\mathcal C(G)$ and $\Ann V$ is a closed ideal in  $\mathcal M_c(G)$. Further, we have
$$
\Ann \Ann I=I\enskip\text{and}\enskip \Ann \Ann V=V
$$
(see \cite[Section 11.6]{MR3185617}, \cite[Section 1]{MR3502634}).
\vskip.2cm

The Fourier--Laplace transformation (shortly: Fourier transformation) on the measure algebra is defined as follows: for every exponential $m$ on $G$ and for each measure $\mu$ in $\mathcal M_c(G)$ its {\it Fourier transform} is 
$$
\widehat{\mu}(m)=\int \widecheck{m}\,d\mu,
$$
where $\widecheck{m}(x)=m(-x)$ for each $x$ in $G$. The Fourier transform $\widehat{\mu}$ is a complex valued function defined on the set of all exponentials on $G$. As the mapping $\mu\mapsto \widehat{\mu}$ is linear and $(\mu*\nu)\,\widehat{}=\widehat{\mu}\cdot \widehat{\nu}$, all Fourier transforms form a function algebra. By the injectivity of the Fourier transform, this algebra is isomorphic to $\mathcal M_c(G)$. If we equip the algebra of Fourier transforms by the topology arising from the topology of $\mathcal M_c(G)$, then we get the {\it Fourier algebra} of $G$, denoted by $\mathcal A(G)$. In fact, $\mathcal A(G)$ can be identified with $\mathcal M_c(G)$. We utilize this identification:  for instance, every ideal in $\mathcal A(G)$ is of the form $\widehat{I}$, where $I$ is an ideal in $\mathcal M_c(G)$. Based on this fact, we say that {\it the ideal $\widehat{I}$ in $\mathcal A(G)$ is synthesizable}, or {\it spectral synthesis holds for the ideal $\widehat{I}$ in $\mathcal A(G)$}, if the variety $\Ann I$ in $\mathcal M_c(G)$ is synthesizable, or spectral synthesis holds for it. 
\vskip.2cm

We shall use the polynomial derivations on the Fourier algebra. A {\it derivation} on $\mathcal A(G)$ is a linear operator $D:\mathcal A(G)\to \mathcal A(G)$ such that
$$
D(\widehat{\mu}\cdot \widehat{\nu})=D(\widehat{\mu})\cdot \widehat{\nu}+\widehat{\mu}\cdot D(\widehat{\nu})
$$
holds for each $\widehat{\mu}, \widehat{\nu}$. We say that $D$ is a {\it first order derivation}. Higher order derivations are defined inductively: for a positive integer $n$ we say that linear operator $D$ on $\mathcal A(G)$ is a {\it derivation of order $n+1$}, if the two variable operator
$$
(\widehat{\mu}, \widehat{\nu})\mapsto D(\widehat{\mu}\cdot \widehat{\nu})-D(\widehat{\mu})\cdot \widehat{\nu}-\widehat{\mu}\cdot D(\widehat{\nu})
$$ 
is a derivation of order $n$ in both variables. The identity operator $id$ is considered a derivation of order $0$. All derivations form an algebra of operators, and the derivations in subalgebra generated by all first order derivations are called  {\it polynomial derivations}. They have the form $P(D_1,D_2,\dots,D_k)$, where $D_1,D_2,\dots,D_k$ are first order derivations, and $P$ is a complex polynomial in $k$ variables. In \cite{Sze23}, we proved the following result:

\begin{Thm}\label{deri}
	The linear operator $D$ on $\mathcal A(G)$ is a polynomial derivation if and only if there exists a unique polynomial $f_D$ such that
	$$
	D\widehat{\mu}(m)=\int f_D(x)m(-x)\,d\mu(x)
	$$
	holds for each $\widehat{\mu}$ in $\mathcal A(G)$ and for every exponential $m$ on $G$. 
\end{Thm}

Iin \cite{Sze23}, we introduced the following concepts. Given an ideal $\widehat{I}$ in $\mathcal A(G)$ and an exponential $m$, we denote by $\mathcal P_{\widehat{I},m}$ the family of all polynomial derivations $P(D_1,D_2,\dots,D_k)$ which annihilate $\widehat{I}$ at $m$. This means that 
$$
\partial^{\alpha}P(D_1,D_2,\dots,D_k)\widehat{\mu}(m)=0
$$
for each multi-index $\alpha$ in $\N^k$, for every exponential $m$, and for every $\widehat{\mu}$ in $\widehat{I}$. The dual concept is the following: given a family $\mathcal P$ of polynomial derivations and an exponential $m$ we denote by $\widehat{I}_{\mathcal P,m}$ the set of all functions $\widehat{\mu}$ which are annihilated by every derivation in the family $\mathcal P$ at $m$. Then $\widehat{I}_{\mathcal P,m}$ is a closed ideal. Obviously, 
$$
\widehat{I}\subseteq \bigcap_m \widehat{I}_{\mathcal P_{\widehat{I},m},m}
$$
holds for every ideal $\widehat{I}$. We call $\widehat{I}$ {\it localizable}, if we have equality in this inclusion. In other words, the ideal $\widehat{I}$ in $\mathcal A(G)$ is localizable if and only if it has the following property: if $\widehat{\mu}$ is annihilated by all polynomial derivations, which annihilate $\widehat{I}$ at each $m$, then $\widehat{\mu}$ is in $\widehat{I}$. The main result in \cite{Sze23} is the following:

\begin{Thm}\label{loc}
	Let $G$ be a locally compact Abelian group. The ideal $\widehat{I}$ in the Fourier algebra is localizable if and only if $\Ann I$ is synthesizable.
\end{Thm}

\begin{Cor}\label{disc}
	Spectral synthesis holds on a discrete Abelian group if and only if its torsion free rank is finite.
\end{Cor}

\begin{proof}
	Although a quite long and complicated proof for this result was given in \cite{MR2340978}, we note that a simple application of Theorem \ref{loc} above gives Theorem 1. and Corollary 3.  in \cite{Sze23b}, which yields our statement. 
\end{proof}

\section{Main result}

We have seen above, that if $G,H$ are discrete Abelian groups and spectral synthesis holds on $G$ and on $H$, then spectral synthesis holds on $G\times H$. On the other hand, by the cited results of D.~I.~Gurevich, the corresponding result may not hold if $G$ and $H$ are non-discrete. Our main result in this paper is the following: if $G$ or $H$ is discrete, then spectral synthesis holds on $G\times H$. In fact, we shall prove the following stronger result:

\begin{Thm}\label{main}
Let $G$ be a locally compact Abelian group and $H$ a closed subgroup of $G$. If $G/H$ is discrete, then spectral synthesis holds on $G$ if and only if spectral synthesis holds on $H$ and on $G/H$.
\end{Thm}

In this situation some authors (see \cite{MR0242802, MR1325242} call $G$ an {\it extension of $H$ by the group $D=G/H$}. To prove this theorem we need a couple of preliminary results. The first one  is interesting on its own. In fact, the corresponding result on discrete Abelian groups was proved in \cite[Lemma 6.]{MR2340978}, but we verify it on general locally compact Abelian groups.

\begin{Lem}\label{subgroup}
	Every closed subgroup of a locally compact synthesizable Abelian group is synthesizable. 
\end{Lem} 

\begin{proof}
	The proof is based on our localization result Theorem \ref{loc}. Namely, we show that if $G$ is a locally compact synthesizable Abelian group and $H$ is a closed subgroup of $G$, then every ideal in the Fourier algebra of $H$ is localizable.
	\vskip.2cm
	
	Let $\widehat{I}$ be an ideal in $\mathcal A(H)$, where $I=\Ann V$ is the annihilator of a variety $V$ on $H$. Clearly, every measure in $\mathcal M_c(H)$ can be considered as a measure in $\mathcal M_c(G)$: in fact, $\mathcal M_c(H)$ is a closed subalgebra of $\mathcal M_c(G)$. It follows that $\mathcal A(H)$ can be considered as a closed subalgebra of $\mathcal A(G)$. We remark that some confusion may arise from the fact that the functions in $\mathcal A(G)$ are defined on the set of exponentials on $G$ and the functions in $\mathcal A(H)$ are defined only on the set of exponentials on $H$ -- but clearly, this minor inconvenience can be overcome by the fact, that every exponential on $H$ cen be extended to an exponential on $G$, and the values of the Fourier transforms of the measures in $\mathcal A(H)$ are independent of the particular extension of the measure, as its support is included in $H$.
	\vskip.2cm
	
	It follows that $\widehat{I}$ is a closed subset in $\mathcal A(G)$. Let $\widehat{I}_G$ denote the closure of the ideal generated by the set $\widehat{I}$ in $\mathcal A(G)$. It follows that the functions of the form
	$\widehat{\xi}\cdot \widehat{\mu}$ with $\widehat{\xi}$ in $\mathcal A(G)$ and $\widehat{\mu}$ in $\widehat{I}$ span a dense subset in $\widehat{I}_G$. We show that if a polynomial derivation $P(D)$ annihilates $\widehat{I}$ at the exponential $m$, then any extension it annihilates $\widehat{I}_G$ at $m$. Indeed, let $P$ be a complex polynomial in $n$ variables and $D=(D_1,D_2,\dots,D_n$, where $D_i$ is a first order derivation on $\mathcal A(H)$. By Theorem \ref{deri}, every polynomial derivation on $\mathcal A(H)$ can be extended to $\mathcal A(G)$. On the other hand, we have the following identity:
	$$
	P(D)(\widehat{\xi}\cdot \widehat{\mu})=\sum_{\alpha\in \N^n} \frac{1}{\alpha!} D^{\alpha}\widehat{\xi}\cdot [\partial^{\alpha}P(D)] \widehat{\mu},
	$$
	which can be verified easily, by Leibniz's Rule. It follows that if $P(D)$ annihilates $\widehat{I}$ at the exponential $m$, then 
	$$
	[\partial^{\alpha}P(D)] \widehat{\mu}(m)=0
	$$
	for each $\alpha$, hence $P(D)$ vanishes on each product $\widehat{\xi}\cdot \widehat{\mu}$ at $m$, which implies that $P(D)$ annihilates $\widehat{I}_G$ at $m$. As $G$ is synthesizable, we have that if $\widehat{\nu}$ is annihilated by each $P(D)$ which annihilates  $\widehat{I}_G$ at each $m$, then $\widehat{\nu}$ is in $\widehat{I}_G$. Now assume that $\widehat{\nu}$ is in $\mathcal A(H)$ and it is annihilated by every polynomial derivation which annihilates $\widehat{I}$ at $m$ in the Fourier algebra $\mathcal A(H)$. Then, considering $\widehat{\nu}$ as a function in $\mathcal A(G)$, is annihilated by every polynomial derivation which  annihilates $\widehat{I}_G$ at $m$, where $m$ denotes  any extension of $m$ to an exponential on $G$. By the above argument, $\widehat{\nu}$ is in $\widehat{I}_G$. But this ovbiously means that $\widehat{\nu}$ is in $\widehat{I}$, as the support of $\nu$ is in $H$. This proves that $\widehat{I}$ is localizable, hence $I=\Ann V$ is synthesizable, and we conclude that spectral synthesis holds on $H$ and our proof is complete.
\end{proof}

We recall the following result as well (see \cite{Sze23e}).

\begin{Lem}\label{contim}
Every continuous homomorphic image of a synthesizable locally compact Abelian group is synthesizable.
\end{Lem}

We also need the following simple known result.

\begin{Lem}\label{embed}
Let $G$ be a locally compact Abelian group and let $H$ be a closed subgroup of $G$. Then $G$ is topologically be embedded in $H\times G/H$.
\end{Lem}

\begin{proof}
	We define the continuous homomorphism $\varphi:H\times G\to G$ by
	$$
	\varphi(h, g)=h+g
	$$for each $h$ in $H$ and $g$ in $G$. Clearly $\varphi$ is surjective, and the First Homomorphism Theorem says that $G$ is topologically isomorphic to $(H\times G)\big/\Ker \varphi$. We note that $\Ker \varphi$ is isomorphic to $H$. Now we define the following continuous homomorphism $\psi: (H\times G)\big/\Ker \varphi\to H\times G/H$ by
	$$
	\psi\bigl((h,g)+\Ker \varphi\bigr)=(h,g+\Ker \varphi)=(0,g+H) 
	$$
	for each $h$ in $H$ and $g$ in $G$. Clearly, $\psi$ is open, and the kernel of $\psi$ is $\{(0,\Ker \varphi)\}$, hence $\psi$ is a topological isomorphism. As $(H\times G)\big/\Ker \varphi$ is topologically isomorphic to $G$, our lemma is proved.
\end{proof}

Our next result is the last step toward the proof of our main result.

\begin{Lem}\label{Zdir}
	If $G$ is a synthesizable locally compact Abelian group, then so is $G\times \Z^k$ for each natural number $k$.
\end{Lem}

\begin{proof}
It is enough to show that if $G$ is a synthesizable locally compact Abelian group, then so is $G\times \Z$. We shall prove this statement in the sequel.
\vskip.2cm

It is known that every exponential $e:\Z\to \mathbb{C}$ has the form
$$
e(k)=\lambda^k
$$
for $k$ in $\Z$, where $\lambda$ is a nonzero complex number, which is uniquely determined by $e$. For this exponential we use the notation $e_{\lambda}$. It follows that for every commutative topological group $G$, the exponentials on $G\times \Z$ have the form $m\otimes e_{\lambda}:(g,k)\mapsto m(g)e_{\lambda}(k)$, where $m$ is an exponential on $G$, and $\lambda$ is a nonzero complex number. Hence the Fourier transforms in $\mathcal A(G\times \Z)$ can be thought as two variable functions defined on the pairs $(m,\lambda)$, where $m$ is an exponential on $G$, and $\lambda$ is a nonzero complex number.
\vskip.2cm

For each measure $\mu$ in $\mathcal M_c(G\times \Z)$ and for every $k$ in $\Z$ we let: 
$$
S_k(\mu)=\{g:\, g\in G\enskip\text{and}\enskip (g,k)\in \supp \mu\}.
$$
This is the $k$-projection of the support of $\mu$ onto $G$. As $\mu$ is compactly supported, there are only finitely many $k$'s in $\Z$ such that $S_k(\mu)$ is nonempty. We have
$$
\supp \mu=\bigcup_{k\in \Z} (S_k(\mu)\times \{k\}),
$$
and
$$
S_k(\mu)\times \{k\}=(G\times \{k\})\cap \supp \mu.
$$
It follows that the sets $S_k(\mu)\times \{k\}$ are pairwise disjoint compact sets in $G\times \Z$, and they are nonempty for finitely many $k$'s only. The restriction of $\mu$ to $S_k(\mu)\times \{k\}$ is denoted by $\mu_{k}$. Then, by definition
$$
\langle \mu_k, f\rangle=\int f\cdot \chi_k\,d\mu
$$
for each $f$ in $\mathcal C(G\times \Z)$, where $\chi_k$ denotes the characteristic function of the set $S_k(\mu)\times \{k\}$. In other words,
$$
\int f\,d\mu_k=\int f(g,k)\,d\mu(g,l)
$$
holds for each $k$ in $\Z$ and for every $f$ in $\mathcal C(G\times \Z)$. Clearly, $\mu= \sum_{k\in \Z} \mu_k$, and this sum is finite. 
\vskip.2cm

Let $\delta_{(0,k)}$ denote the Dirac measure at the point $(0,k)$ in $G\times \Z$. For each $f$ in $\mathcal C(G\times \Z)$, we have:
	$$
	\langle \mu_0*\delta_{(0,k)},f\rangle= \int \int f(g+h,l+n)\,d\mu_0(g,l)\,d\delta_{(0,k)}(h,n)=
	$$
	$$
	\int f(g,l+k)\,d\mu_0(g,l)=\int f(g,k)\,d\mu(g,l)=\langle \mu_k,f\rangle.
	$$

Given a measure $\mu$ in $\mathcal M_c(G\times \Z)$ we define  $\mu_G$ in $\mathcal M_c(G)$ by
$$
\langle \mu_G,\varphi\rangle=\int \varphi(g)\,d\mu(g,l)
$$
whenever $\varphi$ is in $\mathcal C(G)$. Clearly, every $\varphi$ in $\mathcal C(G)$ can be considered as a
function in $\mathcal C(G\times \Z)$, hence this definition makes sense, further we have
$$
\langle \mu_G,\varphi\rangle=\int \varphi(g)\,d\mu_0(g,l).
$$

If $I$ is a closed ideal in $\mathcal M_c(G\times \Z)$, then the set $I_G$ of all measures $\mu_G$ with $\mu$ in $I$, is a closed ideal in $\mathcal M_c(G)$. Indeed, $\mu_G+\nu_G=(\mu+\nu)_G$ and $\lambda\cdot \mu_G=(\lambda\cdot \mu)_G$. Further, if $\mu_G$ is in $I$ and $\xi$ is in $\mathcal M_c(G)$, then we have for each $\varphi$ in $\mathcal C(G)$:
	$$
	\langle \xi*\mu_G,\varphi\rangle=\int \int \varphi(g+h)\,d\xi(g)\,d\mu_G(h)=\int \int  \varphi(g+h)\,d\xi(g)\,d\mu(h,l).
	$$
	On the other hand, we extend $\xi$ from $\mathcal M_c(G)$ to $\mathcal M_c(G\times \Z)$ by the definition
	$$
	\langle \tilde{\xi},f\rangle=\int f(g,0)\,d\xi(g)
	$$
	whenever $f$ is in $\mathcal C(G\times \Z)$. Then 
	$$
	\langle \tilde{\xi}_{G},\varphi\rangle=\int \varphi(g)\,d\tilde{\xi}_0(g,l)=\int \varphi(g)\,d\xi(g)=\langle \xi,\varphi\rangle,
	$$
	that is $\tilde{\xi}_{G}=\xi$. Finally, a simple calculation shows that
	$$
	\langle \xi*\mu_G,\varphi\rangle=\langle (\tilde{\xi}*\mu)_G,\varphi\rangle,
	$$
	hence $\xi*\mu_G=(\tilde{\xi}*\mu)_G$ is in $I_G$, as $\tilde{\xi}*\mu$ is in $I$. This proves that $I_G$ is an ideal in $\mathcal M_c(G)$.
	\vskip.2cm
	
	To show that $I_G$ is closed we assume that $(\mu_{\alpha})$ is a generalized sequence in $I$ such that the generalized sequence $(\mu_{{\alpha},G})$ converges to $\xi$ in $\mathcal M_c(G)$. This means that 
	$$
	\lim_{\alpha} \int \varphi(g)\,d\mu_{{\alpha},G}(g)=\int \varphi(g)\,d\xi(g)
	$$
	holds for each $\varphi$ in $\mathcal C(G)$. In particular, for each exponential $m$ on $G$ we have 
	$$
	\lim_{\alpha} \int \check{m}(g)\,d\mu_{{\alpha},0}(g,l)=\lim_{\alpha} \int \check{m}(g)\,d\mu_{{\alpha},G}(g)=\int \check{m}(g)\,d\xi(g)=\int \check{m}(g)\,d\tilde{\xi}_0(g,l).
	$$
	In other words,
	$
	\lim_{\alpha} \widehat{\mu}_{{\alpha},0}=\widehat{\tilde{\xi}}_0
	$
	holds, which implies
	$
	\lim_{\alpha} \mu_{\alpha,0}=\tilde{\xi}_0,
	$
	consequently
	$$
	\tilde{\xi}_k=\tilde{\xi}_0*\delta_{(0,k)}=\lim_{\alpha} \mu_{{\alpha},0}*\delta_{(0,k)}=\lim_{\alpha} \mu_{{\alpha},k}.
	$$
	Hence we infer
	$$
	\tilde{\xi}=\sum_k \tilde{\xi}_k=\sum_k \lim_{\alpha} \mu_{{\alpha},k}=\lim_{\alpha} \sum_k  \mu_{{\alpha},k}=\lim_{\alpha} \mu_{\alpha},
	$$
	as each sum is finite. Since $I$ is closed, hence $\tilde{\xi}$ is in $I$, which proves that $\tilde{\xi}=\tilde{\xi}_{G}$ is in $I_G$, that is, $I_G$ is closed.
\end{proof}

Now we are ready to prove our main result Theorem \ref{main}.

\begin{proof}
If spectral synthesis holds on the locally compact Abelian group $G$, then it holds on every closed subgroup, by Lemma \ref{subgroup}, hence it holds on $D$, as every discrete subgroup is closed. On the other hand, If spectral synthesis holds on the locally compact Abelian group $G$, then it holds on every continuous homomorphic image of $G$, by Lemma \ref{contim}, hence it holds on $G/D$. This proves the necessity part of the theorem.
\vskip.2cm

Now suppose that the conditions on $G$ and $H$ are satisfied: $H$ is a synthesizable closed subgroup of $G$ and $D=G/H$ is a synthesizable discrete Abelian group. By Corollary \ref{disc}, $D$ has finite torsion free rank, hence it is the continuous homomorphic image of $\Z^k$ with some natural number $k$. By Lemma \ref{subgroup} and Lemma \ref{embed}, it is enough to show that $H\times D$ is synthesizable. As $H\times D$ is a continuous image of $H\times \Z^k$, by Lemma \ref{contim}, it is enough to show that this latter group is synthesizable. But this follows from Lemma \ref{Zdir}.
\end{proof}


\begin{thebibliography}{1}
	\bibitem{MR0023948}
	L.~Schwartz
	\newblock Th\'eorie g\'en\'erale des fonctions moyenne-p\'eriodiques.
	\newblock {\em Ann. of Math. (2)}, 48:857--929, 1947.
	
	\bibitem{MR0242802}
	S.~MacLane and G~Birkhoff,
	\newblock {\em Algebra},
	\newblock The Macmillan Company, New York; Collier Macmillan Ltd., London, 1967.
	
	\bibitem{MR0390759}
	D.~I. Gurevi{\v{c}}.
	\newblock Counterexamples to a problem of {L}. {S}chwartz.
	\newblock {\em Funkcional. Anal. i Prilo\v zen.}, 9(2):29--35, 1975.
	
	\bibitem{MR1325242}
	E.~H.~Spanier,
	\newblock {\em Algebraic topology},
	\newblock Springer-Verlag, New York, 1995.
	
	\bibitem{MR2340978}
	M.~Laczkovich and L.~Sz{\'e}kelyhidi,
	\newblock Spectral synthesis on discrete abelian groups,
	\newblock {\em Math. Proc. Cambridge Philos. Soc.}, {\bf 143(1)} 103--120, 2007.
	
	\bibitem{MR2680008}
	L.~Sz{\'e}kelyhidi,
	\newblock Spectral synthesis problems on locally compact groups,
	\newblock {\em Monatsh. Math.}, {\bf 161(2)} 223--232, 2010.
	
	\bibitem{MR3185617}
	L.~Sz{\'e}kelyhidi.
	\newblock {\em Harmonic and spectral analysis}.
	\newblock World Scientific Publishing Co. Pte. Ltd., Hackensack, NJ, 2014.
	
	\bibitem{MR3502634}
	L.~Sz{\'e}kelyhidi.
	\newblock Annihilator methods for spectral synthesis on locally compact
	{A}belian groups.
	\newblock {\em Monatsh. Math.}, 180(2):357--371, 2016.
	
	\bibitem{Sze23c}
	L.~Sz{\'e}kelyhidi,
	\newblock Characterisation of Locally Compact Abelian Groups Having Spectral Synthesis,
	\newblock 	
	https://doi.org/ 10.48550/ arXiv.2310.19020
	
	\bibitem{Sze23}
	L.~Sz{\'e}kelyhidi,
	\newblock Spectral Synthesis on Varieties,
	\newblock 	
	https:// doi.org/ 10.48550/ arXiv.2306.17438
	
	\bibitem{Sze23b}
	L.~Sz\'ekelyhidi, \rm{New Results on Spectral Synthesis}, to appear in {\emph{ Monatsh. Math.}} https://doi.org/10.1007/s00605-024-01950-6
	
	\bibitem{Sze23e}
		L.~Sz\'ekelyhidi, \rm{Spectral Synthesis on Continuous Images}, to appear in {\emph{Annales Univ. Sci. Budapest., Sect. Comp.}}, https://doi.org/ 10.48550/ arXiv.2405.14261
		
\end{thebibliography}
\end{document}